\documentclass[reqno,11pt]{amsart}
\usepackage[margin=1in]{geometry}

\usepackage[utf8]{inputenc}
\usepackage[T1]{fontenc}

\usepackage{amsthm, amsmath, amssymb,xcolor}

\usepackage[textsize=small,backgroundcolor=orange!20]{todonotes}

\usepackage[hidelinks]{hyperref}

\usepackage[noabbrev,capitalize]{cleveref}
\crefname{equation}{}{}

\usepackage[color,final]{showkeys} 

\colorlet{refkey}{orange!20}
\colorlet{labelkey}{blue!60}

\numberwithin{equation}{section}

\newtheorem{theorem}{Theorem}[section]
\newtheorem{proposition}[theorem]{Proposition}
\newtheorem{lemma}[theorem]{Lemma}

\newtheorem{corollary}[theorem]{Corollary}

\theoremstyle{definition}
\newtheorem{definition}[theorem]{Definition}
\newtheorem{construction}[theorem]{Construction}

\theoremstyle{remark}

\newcommand{\abs}[1]{\left\lvert#1\right\rvert}

\newcommand{\norm}[1]{\left\lVert#1\right\rVert}

\newcommand{\floor}[1]{\left\lfloor #1 \right\rfloor}

\newcommand{\paren}[1]{\left( #1 \right)}

\newcommand{\wh}{\widehat}

\DeclareMathOperator{\core}{core}
\DeclareMathOperator{\tower}{tower}

\newcommand{\EE}{\mathbb{E}}
\newcommand{\FF}{\mathbb{F}}
\newcommand{\RR}{\mathbb{R}}
\newcommand{\PP}{\mathbb{P}}

\newcommand{\cP}{\mathcal{P}}
\newcommand{\cQ}{\mathcal{Q}}
\newcommand{\TV}{\mathrm{TV}}

\title{Removal lemmas and approximate homomorphisms}

\author[Fox]{Jacob Fox}
\address{Department of Mathematics, Stanford University,
Stanford, CA 94305, USA}
\email{jacobfox@stanford.edu}

\author[Zhao]{Yufei Zhao}
\address{Department of Mathematics, Massachusetts Institute of Technology, Cambridge, MA 02139, USA}
\email{yufeiz@mit.edu}

\thanks{Fox was supported by a Packard Fellowship and by NSF award DMS-1855635. Zhao was
supported by NSF Award DMS-1764176, a Sloan Research Fellowship, and the MIT Solomon Buchsbaum Fund.}

\begin{document}

\begin{abstract}
We study quantitative relationships between the triangle removal lemma and several of its variants. One such variant, which we call the \emph{triangle-free lemma}, states that for each $\epsilon>0$ there exists $M$ such that every triangle-free graph $G$ has an $\epsilon$-approximate homomorphism to a triangle-free graph $F$ on at most $M$ vertices (here an \emph{$\epsilon$-approximate homomorphism} is a map $V(G) \to V(F)$ where all but at most $\epsilon \abs{V(G)}^2$ edges of $G$ are mapped to edges of $F$).
One consequence of our results is that the least possible $M$ in the triangle-free lemma grows faster than exponential in any polynomial in $\epsilon^{-1}$. We also prove more general results for arbitrary graphs, as well as arithmetic analogues over finite fields, where the bounds are close to optimal. 
\end{abstract}

\maketitle

\section{Introduction} \label{sec:intro}

\subsection{Graph removal and related results}

The triangle removal lemma of Ruzsa and Szemer\'edi~\cite{RS78} is a fundamental tool in extremal combinatorics.

\begin{theorem}[Triangle removal lemma] \label{thm:TRL}
For every $\epsilon > 0$, there exists $\delta > 0$ such that every $n$-vertex graph with fewer than $\delta n^3$ triangles can be made triangle-free by deleting at most $\epsilon n^2$ edges. 
\end{theorem}

\begin{definition}
	Let $\delta_{TRL}(\epsilon)$ denote the largest possible constant $\delta$ in \cref{thm:TRL}.
\end{definition}

The standard proof of the triangle removal lemma, which uses Szemer\'edi's regularity lemma~\cite{Sze78}, gives an upper bound on $\delta_{TRL}(\epsilon)^{-1}$ which is a tower of $2$'s of height $\epsilon^{-O(1)}$. 
The tower height was improved to $O(\log (1/\epsilon))$ by Fox \cite{Fox11}. 
On the other hand, only a slightly superpolynomial lower bound $1/\delta_{TRL}(\epsilon) \ge (1/\epsilon)^{c\log(1/\epsilon)}$ is known~\cite{RS78}, coming from the Behrend construction of large sets without 3-term arithmetic progressions~\cite{Beh46}.

The standard regularity proof of the triangle removal lemma actually shows that edges can be removed in a bounded complexity way. 

\begin{theorem}[Triangle removal lemma with bounded complexity] \label{thm:TRL-bdd}
For every $\epsilon>0$, there exist $\delta > 0$ and $M$ such that for every $n$-vertex graph $G$ with fewer than $\delta n^3$ triangles, there is a vertex partition $V(G)=V_1 \cup \ldots \cup V_{M}$, and a triangle-free graph $G'$ on $V(G)$ which is complete or empty between each pair $(V_i, V_j)$ and satisfying $|E(G) \setminus E(G')| \leq \epsilon n^2$. 
\end{theorem}

The above formulation of the removal lemma was highlighted by Tao~\cite{Tao06}, who gave a proof of the hypergraph removal lemma with similar bounded complexity features (the hypergraph removal lemma was independently proved by Gowers~\cite{Gow07} and R\"odl and Schacht~\cite{RS09}) and then used it to establish a removal lemma for sparse hypergraphs, which then led to the Gaussian integer analogue of the Green--Tao theorem~\cite{Tao06gaussian} (also see \cite{CFZ15} for an improvement and simplification).

We introduce the notion of an approximate graph homomorphism, 
which allows us to give a succinct restatement of the above result.

\begin{definition}[Approximate homomorphisms]
	Given graphs $G$ and $F$, a map $\phi \colon V(G) \to V(F)$ is an \emph{$\epsilon$-approximate homomorphism} if at most $\epsilon |V(G)|^2$ edges of $G$ do not map to edges of $F$ under $\phi$.
\end{definition}

The usual notion of a graph homomorphism corresponds to $\epsilon = 0$. 
With this notion, \cref{thm:TRL-bdd} is equivalent to the following statement.

\begin{theorem}[Triangle removal lemma with bounded complexity, rephrased] \label{thm:TRL-ah}
For every $\epsilon>0$, there exist $\delta > 0$ and $M$ such that every $n$-vertex graph $G$ with fewer than $\delta n^3$ triangles has an $\epsilon$-approximate homomorphism into some triangle-free graph with at most $M$ vertices.
\end{theorem}

The following special case of \cref{thm:TRL-ah} for triangle-free graphs $G$ is already interesting.

\begin{theorem}[Triangle-free lemma] \label{thm:TFL}
For every $\epsilon > 0$, there exists $M$ such that every triangle-free graph has an $\epsilon$-approximate homomorphism to a triangle-free graph on at most $M$ vertices.
\end{theorem}

\begin{definition}
Let $M_{TFL}(\epsilon)$ denote the smallest possible $M$ in \cref{thm:TFL}.
\end{definition}

Note that the triangle removal lemma (\cref{thm:TRL}) and triangle-free lemma (\cref{thm:TFL}) together imply \cref{thm:TRL-bdd,thm:TRL-ah}.
Indeed, starting with an $n$-vertex graph with fewer than $\delta_{TRL}(\epsilon/2) n^3$ triangles, first delete $(\epsilon/2) n^2$ edges to get rid of all triangles, and then find an $\epsilon/2$-approximate homomorphism into a triangle-free graph on $M_{TRL}(\epsilon/2)$ vertices.

Motivated by graph property testing, Hoppen, Kohayakawa, Lang, Lefmann, and Stagni~\cite{HKLLS20} showed that 
one can deduce \cref{thm:TRL-bdd,thm:TRL-ah,thm:TFL}
using the triangle removal lemma (\cref{thm:TRL}) combined with the Frieze--Kannan weak regularity lemma~\cite{FK99}.
In particular, the deduction does not need the full Szemer\'edi graph regularity lemma.
This implies that
\begin{equation}
	\label{eq:TFL-TRL}
M_{TFL}(\epsilon) \le e^{O(\delta_{TRL}(\epsilon/C)^{-2})},
\end{equation}
which is already better than the usual bound of $M_{TFL} \le \tower(\epsilon^{-O(1)})$ obtained from the standard regularity proof (here $\tower(m)$ denotes an exponential tower of $2$'s of height $m$).
Indeed, \cref{eq:TFL-TRL} is superior since $1/\delta_{TRL}(\epsilon) \le \tower(O(\log(1/\epsilon)))$ \cite{Fox11}, and potentially $1/\delta_{TRL}(\epsilon)$ could be much smaller.
We include a proof sketch of \cref{eq:TFL-TRL} in \cref{sec:TFL-TRL}.

We provide a complementary lower bound to $M_{TFL}(\epsilon)$ in terms of the following close cousin of the triangle removal lemma.

\begin{theorem}[Diamond-free lemma] \label{thm:DFL}
	For every $\epsilon > 0$, there exists some $N$ such that for every $n \ge N$, every $n$-vertex graph where each edge lies in a unique triangle has at most $\epsilon n^2$ edges.
\end{theorem}

\begin{definition}
Let $N_{DFL}(\epsilon)$ denote the smallest constant $N$ so that \cref{thm:DFL} holds.
\end{definition}

The diamond-free lemma is a direct corollary of the triangle removal lemma, yielding $N_{DFL}(\epsilon) \le 1/\delta_{TRL}(\epsilon/3)$. 
Indeed, suppose we have a graph on $n \ge 1/\delta_{TRL}(\epsilon/3)$ vertices and each edge lies in a unique triangle. 
Then the number of triangles is at most a third times the number of edges, which is at most $n^2 \le \delta_{TRL}(\epsilon/3) n^3$.
So by the triangle removal lemma, one can remove at most $(\epsilon/3)n^2$ edges to make this graph triangle-free.
Since the graph was made up of edge-disjoint triangles, it has at most $\epsilon n^2$ edges.

A notable application of the diamond-free lemma is the graph theoretic proof of Roth's theorem on 3-term arithmetic progressions by Ruzsa and Szemer\'edi~\cite{RS09}. 
In fact, this application was one of the original motivations for the triangle removal lemma.
Solymosi~\cite{Sol03} also used the diamond-free lemma to give a short proof of the corners theorem of Ajtai and Szemer\'edi~\cite{AS74}.
The best known lower bound on $N_{DFL}(\epsilon)$ has the form $(1/\epsilon)^{c \log(1/\epsilon)}$, which arises from the Behrend construction of large sets without 3-term arithmetic progressions (for recent improvements on the constant $c$ coming from improved lower bound constructions related to the corners theorem, see \cite{LS,Green-corner}).

Here is a representative case of our main result.
It gives an exponential lower bound for the triangle-free lemma in terms of the bounds in the diamond-free lemma.

\begin{theorem} \label{thm:DFL-TFL}
	There exists a constant $C > 0$ such that, for every $\epsilon > 0$,
	\[
	  M_{TFL}(\epsilon) \ge e^{\epsilon N_{DFL}(C \epsilon)/C}.
	\]
\end{theorem}

Using the best known lower bound on $N_{DFL}(\epsilon)$, we deduce the following superexponential lower bound on $M_{TFL}(\epsilon)$ in terms of $1/\epsilon$.

\begin{corollary} There exists a constant $c>0$ such that for all $0 < \epsilon < 1/2$,
\[
M_{TFL}(\epsilon) 
\ge 
e^{(1/\epsilon)^{c\log(1/\epsilon)}}.
\]
\end{corollary}

We suspect that $N_{DFL}(\epsilon)$ and $1/\delta_{TRL}(\epsilon)$ have similar growth. The next result provides evidence for this suspicion. We show that if $N_{DFL}(\epsilon)$ grows subexponentially in $\epsilon^{-1}$, then $1/\delta_{TRL}(\epsilon)$ does as well.
The proof of the theorem is based on a similar proof in the arithmetic setting by Fox and Lov\'asz \cite{FL17} but uses vertex subset sampling instead of subspace sampling. 

\begin{theorem}\label{theoremdiamondtoremoval}
Fix $0<c<1$. If $N_{DFL}(\epsilon) \leq 2^{\epsilon^{-c+o(1)}}$ as $\epsilon \to 0$, then $\delta_{TRL}(\epsilon) \geq 2^{-\epsilon^{-c/(1-c)+o(1)}}$ as $\epsilon \to 0$. 
\end{theorem}

If $N_{DFL}(\epsilon)$ and $1/\delta_{TRL}(\epsilon)$ have similar growth (as is the case if $N_{DFL}(\epsilon)$ grows subexponentially by Theorem \ref{theoremdiamondtoremoval}), then \cref{thm:DFL-TFL} and the inequality \cref{eq:TFL-TRL} would give comparable lower and upper bounds on $M_{TFL}(\epsilon)$.
Below we also discuss the arithmetic analogue, in which case the best lower and upper bounds indeed match.

Here is the proof strategy for \cref{thm:DFL-TFL}. We start with a graph satisfying the hypotheses of the diamond-free lemma, namely that every edge lies in a unique triangle.
We blow up this graph and then carefully construct a triangle-free subgraph.
By the triangle-free lemma, this final graph we constructed must have an $\epsilon/C$-approximate homomorphism to a triangle-free graph on $M_{TFL}(\epsilon/C)$ vertices, which then implies,  by a novel entropy argument, that the original graph has at most $C\epsilon^{-1} \log M_{TFL}(\epsilon/C)$ triangles.

\medskip

We state below extensions of the triangle removal lemma, the triangle-free lemma, and the diamond-free lemma from a triangle to an arbitrary graph $H$. 
These results are standard in the area, and their proofs use the same techniques as the triangle case.

Although some of these results are commonly stated in terms of $H$-free graphs (with caveats), it will be more natural and relevant for us to discuss them using the following formulations with $H$-homomorphism-free graphs.
We say that a graph $G$ is \emph{$H$-homomorphism-free} if there is no graph homomorphism from $H$ to $G$. A \emph{homomorphic copy} of $H$ in $G$ is a subgraph of $G$ that is the image of a homomorphism from $H$.
The \emph{core} of a graph $H$, denoted $\core(H)$, is defined to be the smallest subgraph of $H$ that can arise as the image of a homomorphism of $H$ (see \cite{HN92}). The core of $H$ is well-defined, i.e., it is unique up to graph isomorphism. 
Indeed, suppose $\phi, \psi \colon H \to H$ are both homomorphisms with images $\phi(H)$ and $\psi(H)$,
then $\psi$ gives a homomorphism from $\phi(H)$ to $\psi(H)$, and vice-versa with $\phi$, so that the two images cannot both be minimal homomorphic copies of $H$ unless they are isomorphic. 
For example, if $H$ is a clique or an odd cycle, then $\core(H) = H$. Also, the core of $H$ consists of a single edge if and only if $H$ is bipartite and has at least one edge.

\begin{theorem} \label{thm:H3}
Let $H$ be a graph. Let $\epsilon > 0$.
	\begin{enumerate}
	\item[(a)]
	There exists $\delta > 0$ such that every $n$-vertex graph with fewer than $\delta n^{\abs{V(H)}}$ homomorphic copies of $H$ can be made $H$-homomorphism-free by removing at most $\epsilon n^2$ edges.  
	
	\item[(b)]
	There exists some $M$ such that every $H$-homomorphism-free graph has an $\epsilon$-approximate homomorphism to an $H$-homomorphism-free graph on at most $M$ vertices.

	\item[(c)]
	Further suppose that $H$ is connected and non-bipartite.
	There exists some $N$ such that for every $n$-vertex graph $G$ with $n \ge N$,
	if every edge of $G$ lies in a unique homomorphic copy of $\core(H)$, then $G$ has at most $\epsilon n^2$ edges.
	\end{enumerate}
\end{theorem}

\begin{definition}
	Let $\delta_H(\epsilon)$, $M_H(\epsilon)$, and $N_H(\epsilon)$ denote the optimal constants $\delta$, $M$, and $N$, respectively, in \cref{thm:H3}.
\end{definition}

Now we state our results comparing the bounds in \cref{thm:H3}, extending the earlier inequality \cref{eq:TFL-TRL} and \cref{thm:DFL-TFL} from triangles to general $H$. 
The lower bound is new.
The upper bound below was already proved in \cite{HKLLS20}, though we sketch a proof in \cref{sec:TFL-TRL}.

\begin{theorem}[Main theorem for graphs] \label{thm:main-graph}
	For every connected non-bipartite graph $H$, there is some constant $C = C_H > 0$ such that, for every $0 < \epsilon < 1$,
	\[
	e^{\epsilon N_{H}(C \epsilon)/C} \le M_H(\epsilon) \le e^{C\delta_H(\epsilon/C)^{-2}}.
	\]
\end{theorem}

\subsection{Arithmetic analogue}

Green~\cite{Gre05} developed an arithmetic analogue of Szemer\'edi's graph regularity lemma and used it to prove the following arithmetic analogue of the triangle removal lemma. 

Let $G$ be an abelian group. Given $X,Y,Z \subseteq G$, a \emph{triangle} in $X \times Y \times Z$ is a triple $(x,y,z) \in X \times Y \times Z$ with $x+y+z = 0$.

\begin{theorem}[Arithmetic triangle removal lemma] \label{thm:arith-TRL}
	For every $\epsilon > 0$, there exists $\delta > 0$ such that for every finite abelian group $G$, and subsets $X,Y,Z \subseteq G$ with fewer than $\delta \abs{G}^2$ triangles in $X\times Y \times Z$, we can remove all triangles by deleting at most $\epsilon \abs{G}$ elements from each of $X,Y,Z$.
\end{theorem}

Green's proof was Fourier analytic.
It was later shown by Kr\'al, Serra, and Vena~\cite{KSV09} that the arithmetic triangle removal lemma actually follows from the triangle removal lemma for graphs and even extends to all groups.

Here is the arithmetic analogue of the diamond-free lemma. It is a corollary of the arithmetic triangle-free lemma. 

\begin{theorem}[Arithmetic diamond-free lemma] \label{thm:arith-DFL}
	For every $\epsilon > 0$, there exists $N$ such that for every finite abelian group $G$ with $\abs{G} \ge N$, and $x_1, \dots, x_l$, $y_1, \dots, y_l$, $z_1, \dots, z_l \in G$ satisfying $x_i + y_j + z_k = 0$ if and only if $i=j=k$, one has $l \le \epsilon \abs{G}$.
\end{theorem}

The sets $\{x_1, \dots, x_l\}$, $\{y_1, \dots, y_l\}$, $\{z_1, \dots, z_l\}$ in \cref{thm:arith-DFL} are commonly known as ``tricolor sum-free sets.''

\medskip

From now on, we restrict to the setting of $G = \FF_p^n$ for a fixed $p$.

\begin{definition}
	Let $\delta_p(\epsilon)$ denote the largest possible constant $\delta$ in \cref{thm:arith-TRL} when restricted to groups of the form $G = \FF_p^n$ for fixed prime $p$.
\end{definition}
	
\begin{definition}
	Let $N_p(\epsilon)$ denote the smallest positive integer so that \cref{thm:arith-DFL} holds when restricted to groups of the form $G = \FF_p^n$ with $p^n \ge N_p(\epsilon)$ and fixed prime $p$.
\end{definition}

In this setting, Green's arithmetic regularity proof of \cref{thm:arith-TRL} also gives us the following stronger statement, analogous of \cref{thm:TRL-bdd,thm:TRL-ah}.

\begin{theorem}[Arithmetic triangle removal lemma with bounded complexity]
    For every $\epsilon > 0$ and prime $p$, there exist $\delta >0$ and a positive integer $m$ such that if $X,Y,Z \subseteq \FF_p^n$ are such that $X\times Y \times Z$ has fewer than $\delta p^{2n}$ triangles, 
    then there exist $X',Y',Z' \subseteq \FF_p^m$ with $X' \times Y' \times Z'$ being triangle-free, 
	and a linear map $\phi \colon \FF_p^n \to \FF_p^m$ such that at most $\epsilon p^n$ elements from each of $X,Y,Z$ do not get mapped to $X',Y',Z'$ respectively.
\end{theorem}

A special case is the following analogue of the triangle-free lemma (\cref{thm:TFL}).

\begin{theorem}[Arithmetic triangle-free lemma] \label{thm:arith-TFL}
	For every $\epsilon > 0$ and prime $p$, there exists a positive integer $m$ such that if $X,Y,Z \subseteq \FF_p^n$ are such that $X\times Y \times Z$ is triangle-free, then there exist $X',Y',Z' \subseteq \FF_p^m$ with $X' \times Y' \times Z'$ being triangle-free, 
	and a linear map $\phi \colon \FF_p^n \to \FF_p^m$ such that at most $\epsilon p^n$ elements from each of $X,Y,Z$ do not get mapped to $X',Y',Z'$ respectively.
\end{theorem}

\begin{definition}
	Let $m_p(\epsilon)$ denote the smallest $m$ in \cref{thm:arith-TFL}. Let $M_p(\epsilon) = p^{m_p(\epsilon)}$.
\end{definition}

Following a breakthrough of Croot, Lev, and Pach~\cite{CLP17} and Ellenberg and Gijswijt~\cite{EG17} on the cap set problem, 
a number of developments together led to the following tight bound on $N_p(\epsilon)$. The upper bound on $N_p(\epsilon)$ was shown by 
Blasiak, Church, Cohn, Grochow, Naslund, Sawin, and Umans~\cite{BCCGNSU17} and independently Alon (unpublished).
The lower bound was first established by Kleinberg and Fu~\cite{FK14} for $p=2$, and then in general by Kleinberg, Sawin, and Speyer~\cite{KSS18} conditional on a conjecture later proved independently by Norin~\cite{Nor19} and Pebody~\cite{Peb18}.

\begin{theorem}[Optimal bounds in arithmetic diamond-free lemma for $\FF_p^n$]
	\label{thm:arith-DFL-bounds}
	For fixed prime $p$, as $\epsilon \to 0$, one has
	\[
		N_p(\epsilon) = \epsilon^{- 1/c_p + o(1)}
	\]
	with constant $0 < c_p <1$ given by
	\begin{equation} \label{eq:cp}
	p^{1-c_p} = \inf_{0 < t < 1} t^{-(p-1)/3}(1 + t + t^2 + \cdots + t^{p-1}).
	\end{equation} 
\end{theorem}

Fox and Lov\'asz~\cite{FL17} proved a polynomial dependence of parameters for the arithmetic triangle removal lemma over $\FF_p^n$, and in fact determined the optimal exponent.

\begin{theorem}[Optimal bounds in arithmetic triangle removal lemma for $\FF_p^n$]
	\label{thm:arith-TRL-bounds}
	For fixed prime $p$, as $\epsilon \to 0$, one has
	\[
		\delta_p(\epsilon) = \epsilon^{1 + 1/c_p + o(1)}
	\]
	where $c_p >0$ is the same constant defined in \cref{thm:arith-DFL-bounds}.
\end{theorem}

We prove the following analogue of \cref{thm:main-graph}.

\begin{theorem} [Main theorem, arithmetic analogue]\label{thm:main-arith}
For any $0 < \epsilon < 1$ and prime $p$,
\[
p^{\epsilon N_p(5\epsilon)/p} \le M_p(\epsilon) \le p^{27\delta_p(\epsilon/4)^{-2}}.
\]
\end{theorem}

\begin{corollary}
	For any fixed prime $p$, as $\epsilon \to 0$,
	\[
		\epsilon^{-1/c_p + 1 + o(1)} 
		\le \log_p M_p(\epsilon)
		\le \epsilon^{-2/c_p - 2 + o(1)}.
	\]
\end{corollary}

One can check that $c_p = (0.172 \cdots + o(1))/\log p$ as $p \to \infty$. 
Indeed, by writing $t = 1 - x/p$ we can deduce that $\lim_{p \to \infty} (\text{RHS of } \cref{eq:cp})/p = \inf_{x > 0} e^{x/3} (1 - e^{-x})/x = e^{-0.172 \cdots}$.
In particular, $c_p = \Theta(1/\log p)$.
So we obtain the following bound.

\begin{corollary}
	There exists a universal constants $C>0$ so that for all $0 < \epsilon < 1/2$ and prime $p$,
	\[
		\epsilon^{ - (\log p)/C} \le \log_p M_p(\epsilon) \le \epsilon^{-C \log p}
	\]
\end{corollary}

For generalizations from triangles to longer cycles in $\FF_p^n$, Lov\'asz and Sauermann \cite{LS19} extended the arithmetic diamond-free lemma with an optimal exponent, and Fox, Lov\'asz, and Sauermann \cite{FLS18} extended the arithmetic removal lemma with a polynomial dependence but left open the optimal exponent.

It is possible to extend the above results from triangles to many other arithmetic patterns (including cycles), though we do not pursue this direction here so as not to further complicate matters. See ~\cite{KSV12,Sha10} for how to deduce removal lemmas for systems of linear equations over $\FF_p$ from graph and hypergraph removal lemmas.

\subsection*{Organization}
In \cref{sec:graph-DFL-TFL} we prove the lower bound in \cref{thm:main-graph}, showing that the triangle-free lemma implies the diamond-free lemma with good bounds, as well as for general $H$.
In \cref{section:diamondsandtriangles}, we prove Theorem \ref{theoremdiamondtoremoval}, which shows that if the diamond-free lemma holds with subexponential bounds, then so does the triangle removal lemma. In \cref{sec:arith-DFL-TFL} we prove the arithmetic analogue of the above, namely the lower bound in \cref{thm:main-arith}, which is based on similar ideas but has a somewhat cleaner execution.
In \cref{sec:TFL-TRL} we prove the upper bounds in \cref{thm:main-graph,thm:main-arith} by showing that, both for the graph version and the arithmetic analogue, the triangle removal lemma and the weak regularity lemma imply the diamond-free lemma with good bounds.

\section{Diamond-free versus triangle-free: graphs} \label{sec:graph-DFL-TFL}

Now we prove the lower bound $e^{\epsilon N_{H}(C\epsilon)/C} \le M_H(\epsilon)$ in \cref{thm:main-graph}. 
Note that being $H$-homomorphism-free is equivalent to being $\core(H)$-homomorphism-free. 
So it suffices to consider $H = \core(H)$, which will be the case for the rest of this section.

\begin{construction}[Partial binary blow-up] \label{con:HG}
Suppose $H = \core(H)$ is connected and has more than one edge. 

Let $G$ be an $n$-vertex graph where every edge is contained in a unique homomorphic copy of $H$. Suppose there are exactly $m$ homomorphic copies of $H$ in $G$, and we enumerate them by $H_1, \dots, H_m$. 
We arbitrarily partition the edge-set of each $H_i$ into two non-empty sets, resulting in $H_i = H_i^{(0)} \cup H_i^{(1)}$.

Let $G'$ be a subgraph of the $2^m$-blow-up of $G$ constructed as follows.
The vertices of $G'$ are indexed by $V(G) \times \{0,1\}^m$. 
For each $i \in [m]$, $s \in \{0,1\}$, and $uv \in E(H_i^{(s)})$, the 
two vertices $(u, x_1, \dots, x_m)$ and $(v, y_1, \dots, y_m)$ in $G'$ are adjacent if $x_i = y_i = s$. These are the only edges in $G'$.
\end{construction}

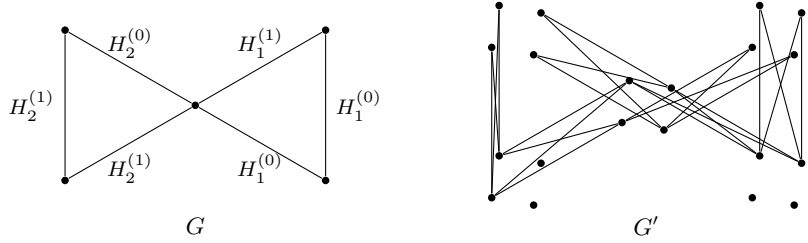
\begin{figure}
\begin{tikzpicture}[v/.style={circle, fill, inner sep = 1pt},scale=2,font=\scriptsize]
\node[v] (a) at (0,0) {};
\node[v] (b) at (30:1) {};
\node[v] (c) at (-30:1) {};
\node[v] (d) at (150:1) {};
\node[v] (e) at (210:1) {};
\draw (a) -- node[above] {$H_1^{(1)}$} (b);
\draw (a) -- node[below] {$H_1^{(0)}$} (c);
\draw (b) -- node[right] {$H_1^{(0)}$} (c);
\draw (a) -- node[above] {$H_2^{(0)}$} (d);
\draw (a) -- node[below] {$H_2^{(1)}$} (e);
\draw (d) -- node[left]  {$H_2^{(1)}$} (e);

\node at (0,-.8) {\footnotesize $G$};

\begin{scope}[xshift=3cm]
\foreach \x/\ss in 
{
	00/(35:.4),
	01/(125:.4),
	10/(-55:.4),
	11/(215:.4)}
{
	\node[v,shift=\ss] (a\x) at (0,0) {};
	\node[v,shift=\ss] (b\x) at (30:1) {};
	\node[v,shift=\ss] (c\x) at (-30:1) {};
	\node[v,shift=\ss] (d\x) at (150:1) {};
	\node[v,shift=\ss] (e\x) at (210:1) {};	
}
\foreach \x in {0,1}
{
	\foreach \y in {0,1}
	{
		\draw (a1\x) -- (b1\y);
		\draw (a0\x) -- (c0\y);
		\draw (b0\x) -- (c0\y);
		\draw (a\x0) -- (d\y0);
		\draw (a\x1) -- (e\y1);
		\draw (d\x1) -- (e\y1);
	}
}
\node at (0,-.8) {\footnotesize $G'$};
\end{scope}
\end{tikzpicture}
\caption{Illustration of the partial binary blow-up, \cref{con:HG}, for $H = K_3$.}
\label{fig:construction}
\end{figure}

See \cref{fig:construction} for an example of the construction.

\begin{lemma} \label{lem:G'-H-hom-free}
	The graph $G'$ obtained in \cref{con:HG} is $H$-homomorphism-free.
\end{lemma}

\begin{proof}
	Suppose we have a homomorphism $\phi \colon H \to G'$.
	We obtain a homomorphism $\psi \colon H \to G$ by composing $\phi$ with the homomorphism $G' \to G$ obtained by projection on the first coordinate of $V(G') = V(G) \times \{0,1\}^m$.
	Since every edge of $G$ lies on a unique homomorphic copy of $H$, $\psi$ must map $H$ to some $H_i$ (notated as in \cref{con:HG}).
	Consider the $i$-th binary coordinate of $\phi(v)$ for $v \in V(H)$.
	This coordinate must equal to $0$ whenever $\psi(v)$ is an endpoint of an edge of $H_i^{(0)}$, and equal to $1$ whenever $\psi(v)$ is an endpoint of an edge of $H_i^{(1)}$. This is impossible to satisfy simultaneously since $H$ is connected. 
\end{proof}

Next, we show that the $G$ constructed above has no $\epsilon$-approximate homomorphism to an $H$-homomorphism-free graph on a small number of vertices.

\begin{proposition}\label{prop:no-approx-hom}
	Suppose $H = \core(H)$ and $\abs{E(H)} > 1$.
	Let $G$ be an $n$-vertex graph where every edge is contained in a unique homomorphic copy of $H$.
	Let $m$ be the number of homomorphic copies of $H$ in $G$.
	Let $G'$ be as in \cref{con:HG}. 
	
	If $\epsilon \le m/(32n^2)$, then there is no $\epsilon$-approximate homomorphism from $G'$ to an $H$-homomorphism-free graph on at most $\exp(c_H m/n)$ vertices, where $c_H>0$ is some constant that depends only on $H$.
\end{proposition}

We first give some intuition for the proof. 
Suppose $\phi \colon V(G') \to V(F)$ is an $\epsilon$-approximate homomorphism and $F$ is $H$-homomorphism-free.
Consider the vertices and edges of $G'$ corresponding to the vertices of some $H_i$, which is a homomorphic copy of $H$ in $G$.
Consider the bipartition $\cP_i$ of $V(H_i) \times \{0,1\}^m \subseteq V(G')$ into two parts separated by the value of the $i$-th binary coordinate. 
If $\phi$ is nearly orthogonal to $\cP_i$ on $V(H_i) \times \{0,1\}^m$ (in the sense that the two associated random variables are nearly independent, as quantified by their mutual information), then the behavior of $\phi$ on $V(H_i) \times \{0,1\}^m$ would be similar to if the construction giving $G'$ had instead used a full $2^m$-blowup of $H_i$ (without taking a subgraph, but with edge weights $1/4$ for normalization).
It would then follow that many edges of $G'$ inside $V(H_i) \times \{0,1\}^m$ cannot map to $F$, since $F$ is $H$-homomorphism-free.

So $\phi$ cannot be nearly orthogonal to too many different $\cP_i$'s. We then show that this would  force its image $V(F)$ to be large. 
To illustrate this argument in an extreme scenario, consider a typical vertex of $G$ that lies in $cm/n$ homomorphic copies of $H$, each of which corresponds to some bipartition $\cP_i$. If $\phi$ were to refine $cm/n$ such $\cP_i$'s, then the image of $\phi$ has size at least $2^{cm/n}$.
We use entropy to give an approximate version of this argument.

\medskip

Given joint discrete random variables $X$ and $Y$, let $H(X)$ denote the (natural base) entropy of $X$, $H(X | Y) = H(X,Y) - H(Y)$ the conditional entropy, and $I(X; Y) = H(X) - H(X|Y)$ their mutual information. 

\begin{definition}
	Let $P_0$ and $P_1$ be two finite disjoint sets of equal size. We say that a non-empty subset $Q \subseteq P_0 \cup P_1$ is \emph{$\eta$-nearly bisected by} $\{P_0, P_1\}$ if the entropy of $\text{Bernoulli}(\abs{Q \cap P_0}/\abs{Q})$ is at least $\log 2 - \eta^2$. 
\end{definition}

Every Bernoulli random variable $W$ satisfies (as can be verified by direct calculation or an application of Pinsker's inequality, e.g., see~\cite{TaoBlog-pinsker})
	\[
	\abs{\PP(W = 0) - \frac{1}{2}} \le \sqrt{\frac{\log 2 - H(W)}{2}}.
	\]
	Thus, every $Q$ that is $\eta$-nearly bisected by $\{P_0, P_1\}$ satisfies 
	\begin{equation}\label{eq:good-reciprocal}
	\abs{\frac{\abs{Q \cap P_0}}{\abs{Q}} - \frac{1}{2}} \le \frac{\eta}{\sqrt{2}}.
	\end{equation}

The next technical lemma says that, if $P_0 \cup P_1$ is a partition with $\abs{P_0} = \abs{P_1}$, and $\cQ$ is another nearly orthogonal partition of the same ground set, then the following two random processes are roughly equivalent: (i) choosing uniform random vertex of $P_0$
and (ii) first choosing a nearly bisected part $Q$ of $\cQ$ with probability proportional to $\abs{Q}$, and then picking a uniform element of $P_0 \cap Q$.

\begin{lemma}\label{lem:good-tv}
	Let $P_0 \cup P_1$ and $Q_1 \cup \cdots \cup Q_k$ be two partitions of some finite set $U$. Suppose $\abs{P_0} = \abs{P_1}$.
	
	Let $u$ be a uniform random element of $U$, and define random variables $X \in \{0,1\}$ and $Y \in [k]$ so that $u \in P_X \cap Q_Y$. Let $\eta < 1/5$. Suppose $	I(X;Y) \le \eta^3$.
	
	Let $J_{\mathrm{nb}} = \{ j \in [k] : Q_j~\text{is}~\eta\text{-nearly bisected by}~\{P_0, P_1\}\}$.
	Let $U_{\mathrm{nb}} = \bigcup_{j \in J_{\mathrm{nb}}} Q_j$. Then $\abs{U_{\mathrm{nb}}} \ge (1-\eta)\abs{U}$.
	
	Choose a random $j \in J_{\mathrm{nb}}$ where each $j \in J_{\mathrm{nb}}$ is chosen with probability proportional to $\abs{Q_j}$.  
	And then choose an element of $P_0 \cap Q_j$ uniformly at random. Let $\mu$ be the distribution of this random element.
	Then the total variation distance between $\mu$ and the uniform distribution on $P_0$ is at most $8\eta$.
\end{lemma}

\begin{proof}
	We have
	\[
		I(X; Y) = H(X) - H(X|Y) = \log 2 - H(X|Y) 
		= \sum_{j=1}^k \PP(u \in Q_j) (\log 2 - H(X | u \in Q_j))
	\]
	Since $H(X | u \in Q_j) < \log 2 - \eta^2$ for every part $Q_j$ which is not $\eta$-nearly bisected by $\{P_0, P_1\}$, the above inequality combined with $I(X; Y) \le \eta^3$ implies
	\begin{equation} \label{eq:Ueq}
		\abs{U_{\mathrm{nb}}} \ge (1-\eta)\abs{U}.		
	\end{equation}
	Then, for any $E \subseteq P_0$,
	\begin{align*}
	\mu(E)
	&= \sum_{j \in J_{\mathrm{nb}}} \frac{\abs{Q_j}}{\abs{U_\mathrm{nb}}} \frac{\abs{E \cap Q_j}}{\abs{P_0 \cap Q_j}} 
	\\
	&= (2\pm 4\eta) \sum_{j \in J_{\mathrm{nb}}} \frac{\abs{E \cap Q_j}}{\abs{U_\mathrm{nb}}} 
	&& \text{\small [by \cref{eq:good-reciprocal}]}
	\\
	&= (2 + \eta \pm 4\eta) \sum_{j \in J_{\mathrm{nb}}} \frac{\abs{E \cap Q_j}}{\abs{U}}.
	&&\text{\small [by \cref{eq:Ueq}]}
	\end{align*}
	If the final sum had been taken over all $j$ (not just $j \in J_{\mathrm{nb}}$), then it would sum to exactly $\abs{E}/\abs{U}$. On the other hand, the $j$'s not in $J_{\mathrm{nb}}$ contribute at most $\eta$ to the sum due to \cref{eq:Ueq}. Thus this sum is at least $\abs{E}/\abs{U} - \eta$. 
	Therefore, $\mu(E)$ differs from $2\abs{E}/\abs{U} = \abs{U}/\abs{P_0}$ by at most $8\eta$, which gives the claimed upper bound on total variance distance.
\end{proof}

\begin{proof}[Proof of \cref{prop:no-approx-hom}]
	Let $\epsilon \leq m/(16n^2)$ and $\phi \colon G' \to F$ be an $\epsilon$-approximate homomorphism where $F$ is $H$-homomorphism-free.
	
	For $v \in V(G)$, let $U_v$ denote the set of vertices in $G'$ of the form $(v, x_1, \dots, x_m)$ for some $x_1, \dots, x_m \in \{0,1\}$.
	Let $U_{v, i \to 0} \subset U_v$ be those vertices with $x_i = 0$, and $U_{v, i \to 1} \subset U_v$ those vertices with $x_i = 1$. Then for each $i\in[m]$, there is a partition $U_v = U_{v, i \to 0} \cup U_{v, i \to 1}$.
	
	For $i \in [m]$ and $v \in V(G)$, write
	\[
		I_{i, v} := I(X; Y)
	\]
	where $X$ is the $i$-th binary coordinate of a uniform random vertex $u \in U_v$ and $Y \in V(F)$ is the image of the same $u$ under $\phi$.
	
	Let $\eta = 1/(32\abs{E(H)})$.
	
	\begin{figure}[t]
	    \centering
	    \begin{tikzpicture}[v/.style={circle, fill, inner sep = 1pt},scale=2,font=\scriptsize]

    \begin{scope}[shift={(-3.5,.5)}]
    \draw (0,0) node[v,label=below:$a$] {}
    	-- node[right] {$H_i^{(1)}$} (60:1) node[v,label=above:$b$] {} 
    	-- node[above] {$H_i^{(0)}$} (120:1) node[v,label=above:$c$] {}
    	-- node[left] {$H_i^{(0)}$} (0,0);
    \node at (0,-.5) {\normalsize $H_i$};
    
    \end{scope}

    \begin{scope}
    \draw (-1,-.5) rectangle (1,.5);
    \draw (0,-.5) -- (0,.5);	
    \draw (.02,.1) ellipse (.7cm and .3cm);
    \node at (-.8,.2) {$Q_{j_a}$};
    \node[v,label=below:$u_{a,0}$] (ua0) at (-.2,.1) {};
    \node[v,label=below:$u_{a,1}$] (ua1) at (.3,.2) {};
    \node at (-.5,-.4) {$U_{a,i \to 0}$};
    \node at (.5,-.4) {$U_{a, i \to 1}$};
    \end{scope}
    
    \begin{scope}[shift=(50:2)]
    \draw (-1,-.5) rectangle (1,.5);
    \draw (0,-.5) -- (0,.5);	
    \draw (.05,-.1) ellipse (.7cm and .3cm);
    \node at (-.7,.1) {$Q_{j_{b}}$};
    \node[v,label=above:$u_{b,0}$] (ub0) at (-.2,-.2) {};
    \node[v,label=above:$u_{b,1}$] (ub1) at (.3,-.1) {};
    \node at (-.5,.4) {$U_{b,i\to 0}$};
    \node at (.5,.4) {$U_{b,i \to 1}$};
    \end{scope}
    
    \begin{scope}[shift=(130:2)]
    \draw (-1,-.5) rectangle (1,.5);
    \draw (0,-.5) -- (0,.5);	
    \draw (-.05,-.1) ellipse (.7cm and .3cm);
    \node at (-.8,.2) {$Q_{j_{c}}$};
    \node[v,label=above:$u_{c,0}$] (uc0) at (-.2,-.3) {};
    \node[v,label=above:$u_{c,1}$] (uc1) at (.3,-.1) {};
    \node at (-.5,.4) {$U_{c,i\to 0}$};
    \node at (.5,.4) {$U_{c,i \to 1}$};
    \end{scope}
    
    \draw (ub0)--(uc0)--(ua0);
    \draw (ua1)--(ub1);
    
    \end{tikzpicture}

    \caption{Illustration for Claim $(\dagger)$ in the proof of \cref{prop:no-approx-hom} with $H = K_3$. The vertices in $Q_{j_a}$ all map to $j_a \in V(F)$ under $\phi$, and likewise with $Q_{j_b}$ and $Q_{j_c}$.}
	    \label{fig:approx-hom-proof}
	\end{figure}
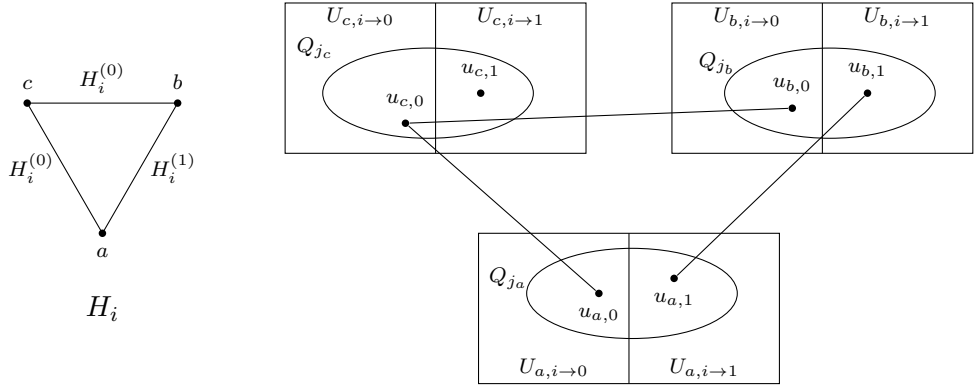
	
	\smallskip
	
	$(\dagger)$ Claim: For a fixed $i$, if $I_{i, v} \le \eta^3$ for all $v \in V(H_i)$, then at least $2^{2m-3}$ edges of $G'$ in $\bigcup_{ab \in E(H_i)} U_a \times U_b$ do not map to an edge of $F$ under $\phi$.
	
	\smallskip

	The reader may find \cref{fig:approx-hom-proof} helpful when following the proof of this claim. 
	The idea is that for each $a \in V(H_i)$ we are going to select a pair of vertices $(u_{a,0}, u_{a,1}) \in U_{a, i \to 0} \times U_{a, i \to 1}$ that agree on $\phi$.
	Then for each $ab \in E(H_i)$, one of $u_{a,0}u_{b,0}$ and $u_{a,1}u_{b,1}$ must be an edge of $G$ (which one depends on whether $ab \in E(H_i^{(0)})$ or $ab \in E(H_i^{(1)})$).
	If all these edges map to edges of $F$ under $\phi$, then we would obtain a homomorphic copy of $H$ in $F$, which is impossible. So one of these edges does not get mapped to an edge of $F$, which then implies the claim by an averaging argument.
	The averaging argument uses that each $u_{a,0}$ (and $u_{a,1}$) is nearly uniformly distributed on its domain by \cref{lem:good-tv}.
	
	Now we proceed with the actual proof.
	Independently for each $a \in V(H_i)$, consider the following process for choosing a pair of vertices $u_{a,0},u_{a,1} \in U_a$. 
	Recall the partition of $U_a$ into $U_{a,i\to 0} \cup U_{a,i \to 1}$ according to the value of the coordinate $x_i$.
	Also partition $U_a$ into $Q_j$'s according to fibers of $\phi$, i.e., set $Q_j = \phi^{-1}(j) \cap U_a$ for each $j \in V(F)$.
	As in \cref{lem:good-tv}, we choose a random part $Q_{j_a}$ that is $\eta$-nearly bisected by $\{U_{a,i\to 0}, U_{a, i \to 1}\}$, where each $Q_{j_a}$ is chosen with probability proportional to $\abs{Q_{j_a}}$. 
	We choose a random vertex $u_{a, 0} \in U_{a,i \to 0} \cap Q_{j_a}$ uniformly at random. Independently, we choose another random vertex $u_{a, 1} \in U_{a,i\to 1} \cap Q_{j_a}$ uniformly at random.
		
	For each $s \in \{0,1\}$ and each $ab \in E(H_i^{(s)})$, consider the edge $u_{a,s}u_{b,s}$ of $G'$ formed by the random vertices chosen earlier (both $u_{a,s}$ and $u_{b,s}$ have their $i$-th binary coordinate equal to $s$, so $u_{a,s}u_{b,s}$ is indeed an edge of $G'$ by \cref{con:HG}).  
	At least one of these $\abs{E(H)}$ edges of $G'$ cannot be mapped to $F$ under $\phi$, or else they would give a homomorphic copy of $H$ in $F$. It follows that
	\[
	\sum_{s \in \{0,1\}} \sum_{ab \in E(H_i^{(s)})} \PP\paren{\phi(u_{a,s})\phi(u_{b,s}) \notin E(F)} \ge 1.
	\]	
	Now choose $u'_{a,0} \in U_{a,i \to 0}$ and $u'_{a,1} \in U_{a,i \to 1}$ independently and uniformly at random for each $a \in V(H_i)$. By \cref{lem:good-tv}, 
	the total variation distance between these random variables satisfies (using the triangle inequality and independence of random variables)
	\begin{align*}
	d_\TV(u_{a,0}u_{b,0}, u'_{a,0}u'_{b,0})
	&\le d_\TV(u_{a,0}u_{b,0}, u'_{a,0}u_{b,0}) + d_\TV(u'_{a,0}u_{b,0}, u'_{a,0}u'_{b,0})
	\\
	&= d_\TV(u_{a,0}, u'_{a,0}) + d_\TV(u_{b,0}, u'_{b,0})
	\le 16\eta.
	\end{align*}
	Thus, combining the above two displayed inequalities,
	\begin{align*}
	\sum_{s \in \{0,1\}} \sum_{ab \in E(H_i^{(s)})} 
	\PP\paren{\phi(u'_{a,s})\phi(u'_{b,s}) \notin E(F)}
	&\ge 
	\sum_{s \in \{0,1\}} \sum_{ab \in E(H_i^{(s)})} 
	(\PP\paren{\phi(u_{a,s})\phi(u_{b,s}) \notin E(F)} - 16\eta)
	\\
	& \ge 1 - 16\abs{E(H)} \eta \ge \frac{1}{2}.
	\end{align*}
	The left-hand side, multiplied by $2^{2m-2}$, equals the number of edges in $\bigcup_{uv \in E(H_i)} U_u \times U_v$ that do not map to $F$ under $\phi$.
	This implies the Claim $(\dagger)$.
	
	\medskip

	For a fixed $v \in V(G)$, choose $X_1, \dots, X_m \in \{0,1\}$ independently and uniformly at random. Let $Y$ be the image under $\phi$ of the vertex $(v, X_1, \dots, X_m)$.
	We have
	\begin{align*}
	\sum_{i=1}^m I_{i,v} 
	 = \sum_{i=1}^m I(X_i; Y)
	& =\sum_{i =1}^m (H(X_i) - H(X_i|Y))
	\\
	&= m \log 2 - \sum_{i =1}^m H(X_i|Y)
	\\
	&\le m \log 2 - H(X_1, \dots, X_m | Y)
	\\
	&= H(Y) \le \log \abs{V(F)}.
	\end{align*}
	Summing over $v \in V(G)$ we obtain
	\[
	\sum_{v \in V(G)} \sum_{i =1}^m I_{i, v} \le n \log \abs{V(F)}.
	\]
	
	Since $\phi$ is an $\epsilon$-approximate homomorphism, at most $\epsilon n^2 2^{2m}$ edges of $G'$ do not map to an edge of $F$.
	Thus the hypothesis of Claim $(\dagger)$ is satisfied for at most $8 \epsilon n^2$ different $i \in [m]$. For all other $i$, one has $I_{i,v} > \eta^3$ for some $v \in V(H_i)$, and thus $\sum_{v \in V(G)} I_{i,v} \ge \eta^3$. Summing over all $i$, we obtain
	\[
	\sum_{v \in V(G)} \sum_{i =1}^m I_{i, v} \ge (m - 8\epsilon n^2)\eta = \frac{m - 8\epsilon n^2}{(32\abs{E(H)})^3} \ge \frac{m}{2(32 \abs{E(H)})^3}.
	\]
	Comparing the above two displayed inequalities, we obtain $\log \abs{V(F)} \ge c_H m/n$, as claimed.
\end{proof}

\begin{proof}[Proof of the lower bound in \cref{thm:main-graph}]
	Let $H$ be connected and non-bipartite and $0 < \epsilon < 1$. We would like to show that if $e^{\epsilon n/C} \ge M_H(\epsilon/C)$, where $C = C_H > 0$ is a sufficiently large constant, then any $n$-vertex graph $G$ where every edge lies in a unique homomorphic copy of $\core(H)$ has at most $\epsilon n^2$ edges.

	Since being $H$-homomorphism-free is equivalent to being $\core(H)$-homomorphism free, we can replace $H$ by its core, and assume from now on that $H = \core(H)$, which has more than one edge since $H$ was originally not bipartite. Suppose for contradiction that the number of homomorphic copies of $H$ in $G$ is $m > \epsilon n^2/\abs{E(H)}$.
	Obtain $G'$ using \cref{con:HG}. 
	Then by \cref{lem:G'-H-hom-free}, $G'$ is $H$-homomorphism-free. 
	Hence by \cref{thm:H3}(b), there exists an $\epsilon/C$-approximate homomorphism from $G'$ to an $H$-homomorphism-free graph on at most $M_H(\epsilon/C) \le e^{\epsilon n/C}$ vertices. 
	On the other hand, by \cref{prop:no-approx-hom}, making sure that $C$ is large enough so that $\epsilon/C \le m/(32n^2)$, there is no $\epsilon$-approximate homomorphism from $G'$ to an $H$-homomorphism-free graph on fewer than $e^{c_H m/n}$ vertices, which contradicts the previous sentence if $C$ is large enough.
\end{proof}

\section{Diamond-free versus triangle removal: graphs} 
\label{section:diamondsandtriangles}
In this section we prove Theorem \ref{theoremdiamondtoremoval}, following the techniques in \cite{FL17}. Assuming that $N_{DFL}(\epsilon)$ grows subexponentially in $\epsilon^{-1}$, it shows that $N_{DFL}(\epsilon)$ and $\delta_{TRL}(\epsilon)$ have similar growth. 

Let $g:(0,1] \longrightarrow \mathbb{R}^+$ satisfy that $g(\beta)$ increases as $\beta$ decreases, $g(\beta)\beta$ decreases as $\beta$ decreases, and $\sum_{i=1}^{\infty} 1/g(2^{-i}) < 1/2$. For example, we may take $g(x)=100\log (100/x)(\log \log (100/x))^2$. 

\begin{lemma} \label{lem1}
Suppose $G$ is a graph on $n$ vertices with $\delta n^3$ triangles and at least $\epsilon n^2$ edges need to be deleted to make $G$ triangle-free. Then $G$ has a subgraph with $\alpha n^3$ triangles for some $0<\alpha \leq \delta$ and no edge is in more than $g(\alpha/\delta)\alpha n/\epsilon$ triangles. 
\end{lemma}
\begin{proof}
We repeatedly delete edges from $G$ one at a time in the most triangles until we arrive at the desired subgraph. Suppose that after removing a certain number of edges, the current remaining subgraph $G'$ has $\beta n^3$ triangles with $\beta \leq \delta$. If no edge is in more than $g(\beta/\delta)\beta n/\epsilon$ triangles in $G'$, then we will see that $G'$ is the desired subgraph as less than $\epsilon n^2$ edges are deleted in total so we have $\beta>0$. Otherwise, we delete the edge in $G'$ in the most triangles. 

To go from $\beta n^3$ triangles to at most $\beta n^3/2$ triangles, we remove at least $g(\beta/(2\delta))(\beta/2)n/\epsilon$ triangles for each edge deleted, so in total we delete at most 
\[
\frac{\beta n^3}{g(\frac{\beta}{2\delta})\frac{\beta n}{2 \epsilon}} = \frac{2\epsilon n^2}{g(\frac{\beta}{2\delta})}
\]edges in halving the total number of triangles from $\beta n^3$ to at most $\beta n^3/2$. In total, we delete at most $\sum_{i=1}^{\infty}2\epsilon n^2/g(1/2^i) < \epsilon n^2$ edges in this process. As the original graph $G$ we assumed required at least $\epsilon n^2$ edges to be deleted to make triangle-free, the remaining subgraph when the process terminates still has at least one triangle and satisfies the desired properties. 
\end{proof}

\begin{lemma}\label{lem2}
Suppose $G$ is a graph on $n$ vertices with $\alpha n^3$ triangles and each edge is in at most $t \leq n/100$ triangles. There is a subgraph of $G$ with $N=n/(9t)$ vertices and more than $\alpha N^3$ edges in which every edge is in exactly one triangle. 
\end{lemma}
\begin{proof}
Pick a random subset $S$ of $N=n/(9t)$ vertices. Call a triangle $T$ of $G$ \emph{good} if it is a subset of $S$ but no edge of $T$ is in another triangle in $S$. The probability $T$ is a subset of $S$ is ${n/(9t) \choose 3}/{n \choose 3} \geq 1/(1000t^3)$. For each triangle $T$, there are at most $3t-3$ other vertices that together with an edge of $T$ make a triangle in $G$. Conditioned on $T$ being a subset of $S$, the probability that another particular vertex is in $S$ is at most $\frac{n/(9t) - 3}{n-3} \leq 1/(9t)$. Thus, conditioning on $T$ is in $S$, the probability that $T$ is good is at least $1-(3t-3)/(9t) > 2/3$. Hence, the expected number of good triangles in $S$ is at least $\frac{1}{1000t^3} \cdot \frac{2}{3} \cdot \alpha n^3 =2 \alpha(n/t)^3/3000$. The edges in the good triangles form a subgraph of $G$ with $N=n/(9t)$ vertices in which each edge is in exactly one triangle and there are at least $\alpha (n/t)^3/500>\alpha N^3$ edges.
\end{proof}

Now we prove \cref{theoremdiamondtoremoval}, which, as a reminder, says that for fixed $0<c<1$, 
if $N_{DFL}(\epsilon) \leq 2^{\epsilon^{-c+o(1)}}$ as $\epsilon \to 0$, then $\delta_{TRL}(\epsilon) \geq 2^{-\epsilon^{-c/(1-c)+o(1)}}$ as $\epsilon \to 0$.

\begin{proof}[Proof of Theorem \ref{theoremdiamondtoremoval}]
Let $g(x)=100\log (100/x)(\log \log (100/x))^2$. 
Let $G$ be a graph on $n$ vertices with $\delta n^3$ triangles such that at least $\epsilon n^2$ edges need to be removed to make $G$ triangle-free. By Lemma \ref{lem1}, $G$ has a subgraph $G'$ with  $\alpha n^3$ triangles for some $0<\alpha \leq \delta$ and no edge is in more than $t:=g(\alpha/\delta)\alpha n/\epsilon$ triangles.
Let $g=g(\alpha/\delta)$. 
So $\alpha/\delta=2^{g^{1-o(1)}}$ as $g \to \infty$.
Also let $\epsilon_0=\epsilon/\left(9g\right)$. 

Applying Lemma \ref{lem2} to $G'$, there is a subgraph $G''$ of $G'$ on $N=n/(9t)=\epsilon_0/\alpha$ vertices with more than $\alpha N^3=\epsilon_0 N^2$ edges and each edge is in exactly one triangle. 

The graph $G''$ shows that $N_{DFL}(\epsilon_0) \geq N=\epsilon_0/\alpha$. On the other hand, by assumption, $N_{DFL}(\epsilon_0) \leq 2^{\epsilon_0^{-c+o(1)}}$ as $g \to \infty$. These two bounds on $N_{DFL}(\epsilon_0)$  together imply $$\epsilon_0^{-1}2^{\epsilon_0^{-c+o(1)}} \geq \alpha^{-1} = (\delta/\alpha)\delta^{-1} =2^{g^{1-o(1)}}\delta^{-1}.$$ 
This bound gives 
\[
\delta^{-1} 
\le 2^{\epsilon_0^{-c+o(1)}-g^{1-o(1)}}
= 2^{(9g/\epsilon)^{c-o(1)}-g^{1-o(1)}}
\le 2^{-\epsilon^{-c/(1-c)+o(1)}}.
\]
The middle term  is maximized when $g=\epsilon^{-c/(1-c)+o_{\epsilon \to 0}(1)}$ and gives the last inequality.
\end{proof}

\section{Diamond-free versus triangle-free in $\FF_p^n$} 
\label{sec:arith-DFL-TFL}

In this section, we prove the lower bound in \cref{thm:main-arith} showing that, in $\FF_p^n$, the triangle-free lemma (\cref{thm:arith-TFL}) implies the diamond-free lemma (\cref{thm:arith-DFL}) with good quantitative bounds.
The idea is to construct a blow-up similar to that done in Section \ref{sec:graph-DFL-TFL} for graphs, though the proof is cleaner here since partitions into cosets are much more rigid than arbitrary partitions.

\begin{construction} \label{con:arith-expand}
	Suppose $x_1, \dots, x_l, y_1, \dots, y_l, z_1, \dots, z_l \in \FF_p^n$ satisfy $x_i + y_j + z_k = 0$ if and only if $i=j=k$.
	
	Let $X'_i$ denote the set of all elements of $\FF_p^{n+l}$ whose first $n$ coordinates form $x_i$, and whose $(n+i)$-th coordinate lies in $\{0, \dots, \floor{(p-2)/3}\}$. Let $X' = \bigcup_{i=1}^l X'_i$.
	
	Let $Y'_i$ denote the set of all elements of $\FF_p^{n+l}$ whose first $n$ coordinates form $y_i$, and whose $(n+i)$-th coordinate lies in $\{0, \dots, \floor{(p-2)/3}\}$. Let $Y' = \bigcup_{i=1}^l Y'_i$.
	
	Let $Z'_i$ denote the set of all elements of $\FF_p^{n+l}$ whose first $n$ coordinates form $z_i$, and whose $(n+i)$-th coordinate lies in  $\{1, \dots, \floor{(p-2)/3}+1\}$. Let $Z' = \bigcup_{i=1}^l Z'_i$.
\end{construction}

Note that $X' \times Y' \times Z'$ above is triangle-free. Indeed, the first $n$ coordinates of any such triangle must form $(x_i, y_i, z_i)$ for some $i$, but then the $(n+i)$-th coordinate cannot sum to zero.

\begin{proposition} \label{prop:arith-missed}
	Let $x_1, \dots, x_l, y_1, \dots, y_l, z_1, \dots, z_l \in \FF_p^n$ and $X', Y', Z' \subset \FF_p^{n+l}$ be as in \cref{con:arith-expand}.
	
	Let $\phi \colon \FF_p^{n+l} \to \FF_p^m$ be any linear map. 
	Let $X'', Y'', Z'' \subset \FF_p^m$.
	Suppose $X'' \times Y'' \times Z''$ is triangle-free.
	Then 
	\[
	\abs{X' \setminus \phi^{-1}(X'')}
	+ \abs{Y' \setminus \phi^{-1}(Y'')}
	+ \abs{Z' \setminus \phi^{-1}(Z'')}
	\ge (l - m)p^l/4.
	\]
\end{proposition}

\begin{proof}
	Since the rank of $\phi$ is at most $m$, there is some $w = (w_1, \dots, w_{n+l})\in \FF_p^{n+l}$ with $\phi(w) = 0$ such that the first $n$ coordinates of $w$ are all zero and $w$ has at least $l-m$ nonzero coordinates.
	Say that $i \in [m]$ is ``good'' if $w_{n+i} \ne 0$.
	
	Fix a good $i$. Writing $X'_i$, $Y'_i$, $Z'_i$ as in \cref{con:arith-expand}, we claim that 
	\begin{equation}
		\label{eq:arith-good}
	\abs{ X'_i \setminus \phi^{-1}(X'')}
	+ \abs{ Y'_i \setminus \phi^{-1}(Y'')}
	+ \abs{ Z'_i \setminus \phi^{-1}(Z'')}
	\ge p^{l-1}(\floor{(p-2)/3}+1) \ge p^l/4.
	\end{equation}
	Summing over all good $i$ yields the claim.
		
	Let us prove \cref{eq:arith-good}.
	Say that $x' \in \FF_p^{n+l}$ lies \emph{above} $x \in \FF_p^n$ if their first $n$ coordinates agree.
	Choose $x', y', z' \in \FF_p^{n+l}$ uniformly at random among triples with $x'+y'+z' = 0$ such that $x',y',z'$ lie above $x_i, y_i, z_i$ respectively and furthermore the $(n+i)$-th coordinates of $x',y',z'$ are all zero.
	
	Let $a,b,c$ be independent uniform random elements from $\{0, \dots, \floor{(p-2)/3}\}$. Multiplying $w$ by a scalar, we may assume that its  $(n+i)$-th coordinate equals to $1$. Let
	\[
	x = x' + aw \in X'_i, \quad y = y'+bw \in Y'_i, \quad \text{and } z = z'+(c+1)w \in Z'_i.
	\]
	Note that $x$ is uniformly distributed in $X'_i$, and likewise with $y$ in $Y'_i$ and $z$ in $Z'_i$.
	
	Since $x'+y'+z'=0$ and $\phi(w) = 0$,  we have $\phi(x) + \phi(y) + \phi(z) = 0$. 
	Due to the hypothesis on $X'', Y'', Z''$, we cannot simultaneously have $\phi(x) \in X''$, $\phi(y) \in Y''$, $\phi(z) \in Z''$. Therefore,
	\[
	\PP(\phi(x) \notin X'')
	+ \PP(\phi(y) \notin Y'')
	+ \PP(\phi(z) \notin Z'')
	\ge 1.
	\]
	Multiplying both sides by $p^{l-1}(\floor{(p-2)/3}+1)$ establishes \cref{eq:arith-good}.
\end{proof}

\begin{proof}[Proof of the lower bound in \cref{thm:main-arith}]
	Suppose $x_1, \dots, x_l, y_1, \dots, y_l, z_1, \dots, z_l \in \FF_p^n$ satisfy $x_i+y_j+z_k=0$ if and only if $i=j=k$. 
	Let $m = m_p(\epsilon/5)$.
	It suffices to show that if $p^n \ge 5m/\epsilon$ then $l \le \epsilon p^n$.
	Indeed, this would imply $N_p(\epsilon)/p \le 5m/\epsilon$ since $N_p(\epsilon)$ is defined to be the smallest possible $p^n$ (with $n$ being a positive integer, which is why we have $N_p(\epsilon)/p$ on the left-hand side) so that we can guarantee the conclusion $l \le \epsilon p^n$.

	Apply \cref{con:arith-expand} to obtain sets $X', Y', Z' \subseteq \FF_p^{n+l}$. Since $X'\times Y' \times Z'$ is triangle-free, by \cref{thm:arith-TFL} there exist $X'',Y'',Z'' \subseteq \FF_p^{m}$ with  $X''\times Y''\times Z''$ triangle-free and a linear map $\phi \colon \FF_p^{n+l} \to \FF_p^m$ such that at most $(\epsilon/5) p^{n+l}$ elements from each of $X',Y',Z'$ do not get mapped to $X'',Y'',Z''$ respectively. 
	On the other hand, \cref{prop:arith-missed} tells us that at least $(l-m)p^l/4$ elements in total from $X',Y',Z'$ combined do not get mapped to $X'',Y'',Z''$ respectively. 
	So $(l-m)p^l/4 \le (\epsilon/5) p^{n+l}$, and hence
	$l \le (4\epsilon/5) p^n+ m \le \epsilon p^n$.
\end{proof}

\section{Triangle-free versus triangle removal} \label{sec:TFL-TRL}

\subsection{Sketch of the argument for graphs}
Here we sketch the proof of upper bound $M_H(\epsilon) \le e^{C\delta_H(\epsilon/C)^{-2}}$ in \cref{thm:main-graph}, which was proved in \cite[Section 3.3]{HKLLS20}. 
In the next subsection, we give the details of the analogous argument in the arithmetic setting.

First one shows that the graph removal lemma \cref{thm:H3}(a) can be extended to allow edge-weights on the $n$-vertex graph with edge-weights in $[0,1]$. When counting $H$ in a weighted graph, we weigh each homomorphic copy of $H$ by the product of the edge-weights.

\begin{theorem}[Weighted graph removal lemma]\label{thm:weighted-removal}
	For every $H$ and $\epsilon > 0$, there exists $\delta > 0$ such that for every $n$-vertex edge-weighted graph $G$ with edge-weights in $[0,1]$, if the weighted number of homomorphisms from $H$ to $G$ is less than $\delta n^{\abs{V(H)}}$, 
	then $G$ can be made $H$-homomorphism-free by removing edges with total weight at most $\epsilon n^2$.	
\end{theorem}

In~\cite{HKLLS20}, the weighted version of the removal lemma was derived from the unweighted version as follows. Starting with a weighted graph $G$, consider the unweighted graph $G'$ consisting of all edges whose edge-weight is at least $\epsilon/2$. If $G$ has $H$-homomorphism-density at most $\delta_H(\epsilon/2) (\epsilon/2)^{\abs{E(H)}}$, then $G'$ has $H$-homomorphism-density at most $\delta_H(\epsilon/2)$, so by the removal lemma, $G'$ can be made $H$-homomorphism-free by removing at most $\epsilon n^2/2$ edges.
Now we remove the same edges from $G$, along with all edges with individual weight less than $\epsilon/2$, and then the resulting weighted graph is $H$-homomorphism-free.

The above argument shows that in \cref{thm:weighted-removal}, one can take $\delta = \delta_H(\epsilon/2) (\epsilon/2)^{\abs{E(H)}}$. 
This is good for most purposes, though we sketch a different argument showing that one can take $\delta = \delta_H(\frac{\epsilon}{\abs{E(H)} + 1})$ in \cref{thm:weighted-removal} (the latter bound is superior when $\delta_H(\epsilon) = \epsilon^{\Theta(1)}$, which is the case if and only if $H$ is bipartite~\cite{Alon02}, but also in the arithmetic analogue below). 
See \cref{lem:weighted-arith-TRL} below for the details of a completely analogous argument in the arithmetic setting.

Let $G$ be a weighted $n$-vertex graph with $H$-homomorphism density less than $\delta = \delta_H(\frac{\epsilon}{\abs{E(H)} + 1})$.
Randomly blow $G$ up to an $mn$-vertex graph $G'$. This means replacing every edge $xy \in E(G)$ with edge-weight $w(x,y)$ by a random bipartite graph with $m$ vertices in each part and random edges appearing independently with probability $w(x,y)$. We view $G$ as fixed and consider $m \to \infty$.
Then with probability $1-o(1)$, the $H$-homomorphism density in $G$ is less than $\delta$. 
So by the graph removal lemma (\cref{thm:H3}(a)), one can delete at most $\epsilon m^2n^2/(\abs{E(H)} + 1)$ edges from $G'$ to make it $H$-homomorphism-free. 
For each edge $xy$ of $G$, delete it from $G$ if more than $w(x,y) m^2/(\abs{E(H)} + 1)$ edges sitting above it were deleted from $G'$.
This then deletes edges from $G$ with total weight at most $\epsilon n^2$. Furthermore, with probability $1-o(1)$, no homomorphic copy of $H$ remains. Indeed, suppose some homomorphic copy $H_0$ of $H$ were to remain. Consider a random copy of $H_0$ in $G'$ above $H_0$. A linearity of expectations argument (here we use that with high probability all edges of $G'$ between the same pair of parts lie in roughly the same number of copies of $H_0$, as can be verified by a Chernoff bound argument) shows that with positive probability one of these copies of $H_0$ does not contain any deleted edges, which violates that we had deleted edges from $G'$ and made it $H$-homomorphism-free.

Now we sketch the argument in \cite{HKLLS20} that derives \cref{thm:H3}(b) from \cref{thm:H3}(a) yielding the bound $M_H(\epsilon) \le e^{C \delta_H(\epsilon/C)^{-2}}$ in \cref{thm:main-graph}. 
Starting with an $H$-homomorphism-free graph $G$, we can apply the Frieze--Kannan weak regularity lemma to obtain a $\delta/C$-weak-regular partition $\cP$ of $G$ with $M = e^{O(\delta^{-2})}$ parts. Let $G/\cP$ be the corresponding reduced weighted graph whose vertices are the parts of the partition and weights being the edge densities between the corresponding pairs of parts.
By the counting lemma, the $H$-homomorphism-densities in $G$ and $G/\cP$ differ by $O_H(\delta/C)$.
We can choose the constant $C$ so that $G/\cP$ has $H$-homomorphism density at most $\delta$. 
Then \cref{thm:H3}(a) allows us to make the reduced graph $H$-homomorphism-free by removing weighted edges in $G/\cP$ corresponding to at most $\epsilon n^2$ edges in $G$.
Then the map from $V(G)$ to $\cP$ gives an $\epsilon$-approximate homomorphism from $G$ to an $H$-homomorphism-free graph with $M$ parts.

\subsection{Arithmetic analogue}

Now we provide the arithmetic analogue of the argument sketched above, thereby showing the upper bound $M_p(\epsilon) \le p^{27\delta_p(\epsilon/4)^{-2}}$ in \cref{thm:main-arith}.

Given a function $f \colon \FF_p^n \to \RR$, and a subspace $H \le \FF_p^n$, we write $f_H \colon \FF_p^n \to \RR$ for the function that is constant on every $H$-coset, so that on $x+H$ the value of $f_H$ equals to the average of $f$ on $x+H$. We say that $H$ is \emph{$\epsilon$-weakly-regular} for $f$ the $L^\infty$ norm of the Fourier transform of $f - f_H$ is at most $\epsilon$.
We say that $H$ is $\epsilon$-weakly regular for a set $X \subseteq \FF_p^n$ if it is so for its indicator function $f = 1_X$. 
Here the normalization of the Fourier transform is given by $\wh f(y) = \EE_x f(x) e^{-2\pi i (x \cdot y)/p}$.
Also we write $\|f\|_1 = \EE_x \abs{f(x)}$.

We recall the weak regularity lemma and the associated counting lemma, both of which are standard (e.g., see \cite[Section 2]{FP}). These are versions of Green's arithmetic regularity results~\cite{Gre05} analogous to the Frieze--Kannan weak regularity lemma~\cite{FK99}.

\begin{lemma}[Weak arithmetic regularity lemma] \label{lem:arith-wk-reg}
	Let $p$ be a prime and $\epsilon > 0$.
	For every $X,Y,Z \subseteq \FF_p^n$, there exists a subspace $H$ of $\FF_p^n$ of codimension at most $3 \epsilon^{-2}$ that is $\epsilon$-weakly-regular for each of $X,Y,Z$.
\end{lemma}

A quick proof sketch: take $H$ to be the subspace orthogonal to all non-trivial characters with Fourier transform magnitude at least $\epsilon$ for any of $X,Y,Z$. There are at most $\epsilon^{-2}$ such characters for $X$ by Parseval, and likewise with $Y$ and $Z$. 

For $f,g,h \colon \FF_p^n \to [0,1]$, let us denote their triangle density by
\[
\Lambda(f,g,h) := \EE_{x,y, z\in \FF_p^n : x+y+z=0} f(x)g(y) h(z).
\]
The following counting lemma is also standard (e.g., see \cite[Lemma 4]{FP} for a proof).

\begin{lemma}[Counting lemma] \label{lem:arith-counting}
	Let $p$ be a prime and $\epsilon > 0$.
	For every $f,g,h \colon \FF_p^n \to [0,1]$ and a subspace $H$ of $\FF_p^n$ that is $\epsilon$-regular with respect to each of $f,g,h$, then
	\[
	\abs{\Lambda(f,g,h) - \Lambda(f_H,g_H,h_H)}\le 3\epsilon.
	\]
\end{lemma}

We need the following weighted version of the arithmetic triangle removal lemma. The proof follows the second argument sketched in the previous subsection (the first argument sketched there, by considering edges with weight at least $\epsilon/2$, would be too lossy).

\begin{theorem}[Weighted arithmetic triangle removal lemma] \label{lem:weighted-arith-TRL}
	If $f,g,h \colon \FF_p^n \to [0,1]$ are such that $\Lambda(f,g,h) < \delta_p(\epsilon/4)$,
	 then there exist $f',g,h' \colon \FF_p^n \to [0,1]$ such that $\Lambda(f',g',h') = 0$ and $\norm{f-f'}_1$, $\norm{g-g'}_1$, $\norm{h-h'}_1 \le \epsilon$.
\end{theorem}

\begin{proof}
	In this proof, we fix $f,g,h \colon \FF_p^n \to [0,1]$ with $\Lambda(f,g,h) < \delta_p(\epsilon/4)$.
	All the asymptotics are with respect to a new parameter $m \to \infty$.
	
	We say that $y \in \FF_p^{n+m}$ is \emph{above} $x \in \FF_p^n$ if the first $n$ coordinates of $y$ form $x$.
	
	Let $X$ be a random subset of $\FF_p^{n+m}$ obtained by independently keeping each element above every $x \in \FF_p^n$ with probability $f(x)$.
	Likewise define $Y,Z \subseteq \FF_p^{n+m}$ from $g,h$ respectively.
	
	With high probability (meaning probability $1-o(1)$ as $m\to \infty$), $X \times  Y \times Z$ has at most $\delta_p(\epsilon/4) p^{2(n+m)}$ triangles, so by the arithmetic triangle removal lemma (\cref{thm:arith-TRL}), we can remove all triangles by deleting at most $\epsilon p^{n+m}/4$ elements from each of $X,Y,Z$.
	
	For each $x \in \FF_p^n$, we set $f'(x) = 0$ if we deleted least $f(x) p^m/4$ elements of $X$ above $x$, and set $f'(x) = f(x)$ otherwise. Then the number of elements deleted from $X$ is at least $\sum_{x \in \FF_p^n} (f-f')(x)p^m/4$. 
	Thus $\norm{f-f'}_1 \le \epsilon$. 
	Similarly define $g'$ and $h'$.
	
	Finally, we claim with high probability, $\Lambda(f',g',h') = 0$. 
	Suppose otherwise. Fix some $x+y+z = 0$ in $\FF_p^n$ with $f'(x),g'(y),h'(z) > 0$.
	Among all triples $(x',y',z') \in X\times Y \times Z$ sitting above $(x,y,z)$ and satisfying $x'+y'+z' = 0$, choose a triple uniformly at random. One of $x',y',z'$ must be deleted to make $X,Y,Z$ triangle-free, so
	\[
	\PP(x' \text{ is deleted})
	+	\PP(y' \text{ is deleted})
	+	\PP(z' \text{ is deleted}) \ge 1.
	\] 
	On the other hand, the total variation distance between $x'$ and a uniform random element of $X$ above $x$ is $o(1)$ with high probability (e.g., a second moment argument shows that almost all $x'$ lies in nearly the same number of such triples $(x',y',z')$). So if $r_x$ is the fraction of elements above $x$ that are deleted, then $r_x + r_y + r_z \ge 1-o(1)$ with high probability, thereby contradicting $r_x,r_y,r_z \le 1/4$.
\end{proof}

\begin{proof}[Proof of the upper bound in \cref{thm:main-arith}]
	We want to show that the arithmetic triangle-free lemma (\cref{thm:arith-TFL}) holds with some $m \le 27 \delta^{-2}$, where $\delta := \delta_p(\epsilon/4)$.
	
	Let $X,Y,Z \le \FF_p^n$ be such that $X\times Y \times Z$ is triangle-free. 
	Applying the arithmetic weak regularity lemma (\cref{lem:arith-wk-reg}), we find a subspace $H$ of $\FF_p^n$ with codimension $m \le 27 \delta^{-2}$ that is $\delta/3$-weakly-regular to each of $X,Y,Z$. 
	
	Let $\phi \colon \FF_p^n \to \FF_p^m$ be a linear map with kernel $H$. 
	Define $f,g,h \colon \FF_p^m \to [0,1]$ by setting, for each $x \in \FF_p^m$, $f(x) = \abs{\phi^{-1}(x) \cap X}/p^{n-m}$. In other words, $f(x)$ is the fraction of the coset $H + \phi^{-1}(x)$ that belongs to $X$. Likewise define $g$ and $h$ based on $Y$ and $Z$.
	
	Applying the counting lemma (\cref{lem:arith-counting}),
	\[
	\Lambda(f,g,h) = \Lambda((1_X)_H,(1_Y)_H,(1_Y)_H)
	\le \Lambda(1_X,1_Y,1_Y) + \delta
	= \delta.
	\]
	By the weighted arithmetic triangle removal lemma, \cref{lem:weighted-arith-TRL}, there are $f',g',h' \colon \FF_p^m \to [0,1]$ so that $\Lambda(f',g',h') = 0$ and $\norm{f-f'}_1$, $\norm{g-g'}_1$, $\norm{h-h'}_1 \le \epsilon$. The conclusion of \cref{thm:arith-TFL} follows then by taking $X',Y',Z'$ to be the respective supports of $f',g',h'$.
\end{proof}

\subsection*{Acknowledgments} We thank Yuval Wigderson and the anonymous reviewer for careful readings and comments on the manuscript. 


\end{document}